\documentclass[11pt,
reqno]{amsart}

\usepackage{amsmath,amssymb,amscd,amsfonts,amsthm}
\usepackage{graphics}
\usepackage{dsfont}
\usepackage[all]{xy}
\usepackage{enumerate}

{\catcode`\@=11
\gdef\n@te#1#2{\leavevmode\vadjust{%
 {\setbox\z@\hbox to\z@{\strut#1}%
  \setbox\z@\hbox{\raise\dp\strutbox\box\z@}\ht\z@=\z@\dp\z@=\z@%
  #2\box\z@}}}
\gdef\leftnote#1{\n@te{\hss#1\quad}{}}
\gdef\rightnote#1{\n@te{\quad\kern-\leftskip#1\hss}{\moveright\hsize}}
\gdef\?{\FN@\qumark}
\gdef\qumark{\ifx\next"\DN@"##1"{\leftnote{\rm##1}}\else
 \DN@{\leftnote{\rm??}}\fi{\rm??}\next@}}
\DeclareFontFamily{OT1}{wncyr}{\hyphenchar\font45 }
\DeclareFontShape{OT1}{wncyr}{m}{n}{%
   <5> <6> <7> <8> <9> gen * wncyr
   <10> <10.95> <12> <14.4> <17.28> <20.74>  <24.88>wncyr10}{}
\DeclareFontShape{OT1}{wncyr}{m}{it}{%
   <5> <6> <7> <8> <9> gen * wncyi
   <10> <10.95> <12> <14.4> <17.28> <20.74> <24.88> wncyi10}{}
\DeclareFontShape{OT1}{wncyr}{m}{sc}{%
   <5> <6> <7> <8> <9> <10> <10.95> <12> <14.4>
   <17.28> <20.74> <24.88>wncysc10}{}
\DeclareFontShape{OT1}{wncyr}{b}{n}{%
   <5> <6> <7> <8> <9> gen * wncyb
   <10> <10.95> <12> <14.4> <17.28> <20.74> <24.88>wncyb10}{}
\input cyracc.def
\def\rus{\usefont{OT1}{wncyr}{m}{n}\cyracc\fontsize{9}{11pt}\selectfont}

\DeclareMathSizes{9}{9}{7}{5}

\theoremstyle{plain}

\newtheorem{theorem}{Theorem}

\newtheorem{proposition}{Proposition}
\newtheorem{lemma}{Lemma}
\newtheorem{corollary}{Corollary}

\theoremstyle{definition}

\newtheorem{conjecture}{Conjecture}
\newtheorem{definition}{Definition}
\newtheorem{remark}{Remark}

\newtheorem{nothing*}[theorem]{}
\newtheorem{subnothing*}[sub]{}
\newtheorem{example}{Example}

\newtheorem{question}{Question}

\theoremstyle{remark}

\newcommand{\lb}{{\rm(\hskip -.1mm}}
\newcommand{\rb}{{\hskip .13mm\rm)}}

\newcommand{\dss}{\hskip -2mm\rotatebox{68}{\raisebox{-1.8\height}{\mbox{\normalsize -\hskip .1mm-\hskip .1mm-}}}\hskip -.6mm}

\begin{document}

\title[Jordan groups and automorphism groups of algebraic varieties]
{Jordan groups and\\ automorphism groups of\\ algebraic varieties}

\author[Vladimir  L. Popov]{Vladimir  L. Popov${}^*$}

\thanks{
 ${}^*$\,Supported by
 grants {\rus RFFI
11-01-00185-a}, {\rus N{SH}--5139.2012.1}, and the
program {\it Contemporary Problems of Theoretical
Mathematics} of the Russian Academy of Sciences, Branch
of Mathematics. }

\maketitle

\vskip -10mm

\begin{gather*}
\mbox{\footnotesize Steklov Mathematical Institute, Russian Academy of Sciences}\\[-2mm]
\mbox{\footnotesize Gubkina 8, Moscow 119991, Russia
}\\[-1.5mm]
\mbox{\footnotesize and}\\[-1.5mm]
\mbox{\footnotesize National Research University Higher School of Economics}\\[-2mm]
\mbox{\footnotesize 20, Myasnitskaya Ulitsa, Moscow 101000, Russia}\\
\mbox{\footnotesize popovvl@mi.ras.ru}
\end{gather*}

\begin{abstract} The first section of this paper is focused on Jordan groups in abstract setting, the second on that in the settings of automorphisms groups and groups of birational self-maps
of algebraic varieties.\;The ap\-pen\-dix contains formulations of some open problems and the relevant
comments.

\vskip 2mm


\end{abstract}

\vskip 10mm

This is the expanded version of my talk, based on \cite[Sect.\;2]{Po1},
at the workshop {\it Groups of Automorphisms in Birational and Affine Geometry}, October 29--Novem\-ber 3, 2012, Levico Terme, Italy.\;The ap\-pen\-dix is the expanded version of my notes on open problems posted on the site
of this workshop \cite{Po4}.

Below $k$ is an algebraically closed field of characteristic zero.\;Variety
means algebraic variety over $k$  in the sense of Serre
(so algebraic group means algebraic group over $k$).\;We use without explanation standard notation and conventions
of \cite{Bo} and \cite{Sp}.\;In particular,
$k(X)$ denotes the field of rational functions of an irreducible variety
$X$.\;${\rm Bir}(X)$ denotes the group of birational self-maps
of an irreducible variety $X$. Recall that if $X$ is the affine $n$-dimensional space ${\bf A}^n$,  then
${\rm Bir}(X)$ is called
the {\it Cremona group over} $k$
{\it of rank} $n$; we denote it by ${\rm Cr}_n$ (cf.\,\cite{P3}, \cite{P2}).

\vskip 2mm


I am
indebted to
the referee for careful reading and remarks.

\section{Jordan groups}
\subsection{Main definition}  The notion of Jordan group was introduced in \cite{Po1}:
\begin{definition}[{{\rm \cite[Def.\;2.1]{Po1}}}]\label{Jd}
A group $G$ is called a {\it Jordan group}
if there exists a positive integer $d$,
depending on $G$ only,
such that every finite subgroup
$K$ of $G$ contains a normal abelian subgroup whose index in $K$ is at most $d$.\;The minimal such $d$ is called the {\it Jordan constant of} $\,G$ and is denoted by $J_G$.
\end{definition}

Informally, this means that all finite subgroups of $G$ are ``almost'' abelian in the sense that they are extensions of abelian groups by finite groups taken from a finite list.

Actually, one obtains the same class of groups if the assumption of normality in Definition \ref{Jd} is dropped.\;Indeed,
for any group $P$ containing a
subgroup $Q$
of finite index,
there is
a normal subgroup $N$ of $P$
such that $[P:N]\leqslant [P:Q]!$ and $N\subseteq Q$
(see, e.g.,\,\cite[Exer.\,12 to Chap.\,I]{Lang}).

\subsection{Examples}

\subsubsection{Jordan's Theorem}\label{JJJJJ}
The first example that led to Definition\;\ref{Jd} justifies the coined name. It is given by the classical Jordan's theorem \cite{Jo} (see, e.g., \cite[\S36]{CR} for a modern exposition).\;In terms of
Definition\;\ref{Jd} the latter can be re\-for\-mu\-lated  as follows:

\begin{theorem}[{{\rm C.\;Jordan, 1878}}]\label{Jt}
The group $\,{\bf GL}_n(k)$ is Jordan for every\;$n$.
\end{theorem}

 Since the symmetric group ${\rm Sym}_{n+1}$ admits a faithful $n$-dimensional representation and the alternating group ${\rm Alt}_{n+1}$ is the only non-identity proper normal subgroup of ${\rm Sym}_{n+1}$ for
 $n\geqslant 2$, $n\neq 3$,
Definition \ref{Jd} yields the lower bound
\begin{equation}\label{J}
(n+1)!\leqslant J_{{\bf GL}_n(k)}\quad \mbox{for $n\geqslant 4$}.
\end{equation}

Frobenius, Schur, and Blichfeldt initiated exploration of the upper bounds for  $J_{{\bf GL}_n(k)}$.\;In 2007, using the classification of finite
simple groups, M.\,J.\,Col\-lins \cite{Co} gave optimal upper bounds  and thereby found the
precise values of
$J_{{\bf GL}_n(k)}$ for all $n$.\;In particular, in
\cite{Co} is proved that

\begin{enumerate}[\hskip 2.2mm\rm(i)]
\item the equality in \eqref{J} holds
for all $n\geqslant 71$ and $n=63, 65, 67, 69$;
\item $J_{{\bf GL}_n(k)}=60^rr!$ if $n=2r$ or $2r+1$ and either $20\leqslant n \leqslant 62$ or $n=64, 66, 68, 70$;
    \item $J_{{\bf GL}_n(k)}=60, 360, 25920, 25920, 6531840$\;\;resp., for\;\,$n=2, 3, 4, 5, 6$.
\end{enumerate}
The values of $J_{{\bf GL}_n(k)}$ for $7 \leqslant n\leqslant 19$ see in \cite{Co}.

\subsubsection{Affine algebraic groups}

Since any subgroup of a Jordan groups is Jordan,
Theorem \ref{Jt} yields
\begin{corollary}  Every linear group is Jordan.
\end{corollary}

Since every affine algebraic group
is linear
\cite[2.3.7]{Sp}, this, in turn, yields the following generalization of Theorem \ref{Jt}:
\begin{theorem}\label{lg}
 Every affine algebraic group is Jordan.
\end{theorem}

\subsubsection{Nonlinear Jordan groups}

Are there nonlinear Jordan groups?\;The next example, together with Theorem \ref{Jd}, convinced me that Definition\;\ref{Jd}
singles out an interesting class of groups
and therefore
deserves to be introduced.

\begin{example}\label{Cr2} By \cite[Thm.\;5.3]{Se1}, \cite[Thm.\;3.1]{Se2},
 the planar Cremona group ${\rm Cr}_2$ is Jordan.\;On the other hand, by  \cite[Prop.\;5.1]{CD} (see also \cite[Prop.\;2.2]{Co}), ${\rm Cr}_2$ is not linear.\;Note
that  in \cite[Thm.\;5.3]{Se1} one also
finds a ``multiplicative'' upper bound for $J_{{\rm Cr}_2}$: as is
specified
there,
 a crude computation shows that every finite subgroup $G$ of   ${\rm Cr}_2$ contains a normal abelian subgroup $A$ of rank $\leqslant 2$
 with $[G:A]$
  dividing $2^{10}\cdot 3^4\cdot 5^2\cdot 7$ (it is also mentioned that the exponents of $2$ and $3$ can be somewhat lowered, but those of $5$ and $7$ cannot).
\quad $\square$ \end{example}

\begin{example}\label{Ad} Let $F_d$ be a free group with $d$ free generators and let $F_d^n$ be its normal subgroup generated by the $n$th powers of all elements. As is known (see, e.g.,\;\cite[Thm.\;2]{Ad1}), the group $B(d, n):=F_d/F_d^n$ is  infinite for $d\geqslant 2$ and odd $n\geqslant 665$
(recently S.\,Adian announced in \cite{Ad2}
that $665$ may be replaced by $100$).\;On the other hand, by I.\;Schur, finitely generated linear torsion groups are finite (see, e.g.,\;\cite[Thm.\;36.2]{CR}).\;Hence
infinite $B(d, n)$ is nonlinear.\;On the other hand, for $d\geqslant 2$ and odd $n\geqslant 665$,
every finite subgroup in
$B(d, n)$ is cyclic (see \cite[Thm.\;8]{Ad1}); hence  $B(d, n)$ is Jordan and $J_{B(d, n)}=1$.
\quad $\square$ \end{example}

\begin{example}  Let $p$ be a positive prime integer and let
$T(p)$ be a Tarski monster group, i.e., an infinite group, such that every its proper subgroup is a cyclic group of order $p$. By \cite{Ol83},  for big $p$ (e.g., $\geqslant 10^{75}$),
such a group exists. $T(p)$ is necessarily simple and finitely generated (and, in fact, generated by every two non-commuting elements). By the same reason as in Example \ref{Ad}, $T(p)$ is not linear. The definitions imply that $T_p$ is Jordan and $J_{T(p)}=1$.
(I thank A.\;Yu.\;Ol'shanski\v i who drew  in \cite{Ol13} my attention to this example.)
\quad $\square$ \end{example}

\subsubsection{Diffeomorphism groups of smooth topological manifolds}\label{manif} Let $M$ be a compact connected $n$-dimensional smooth topological manifold.\;Assume that $M$ admits an unramified covering ${\widetilde M}\to M$ such that $H^1({\widetilde M}, \bf Z)$ contains the cohomolo\-gy classes $\alpha_1,\ldots, \alpha_n$  satisfying $\alpha_1\cup\cdots\cup\alpha_n\neq 0$.\;Then, by \cite[Thm.\;1.4(1)]{MR}, the gro\-up ${\rm Diff}(M)$ is Jordan.\;This result  is
applicable to ${\bf T}^n$, the product of $n$ circles, and, more generally, to the connected sum $N\sharp{\bf T}^n$, where $N$ is any compact connected orientable smooth topological manifold.\;(I thank  I.\;Mundet i Riera who
drew in \cite{MR13} my attention to
\cite{MR}, \cite{Fi} and \cite{Pu}.)

\subsubsection{Non-Jordan groups}

Are there non-Jordan groups?

\begin{example}\label{nJ1}
The group
${\rm Sym}_\infty$
of all permutations of $\mathbf Z$ contains  the alternating group ${\rm Alt}_n$ for every $n$. Hence ${\rm Sym}_\infty$ is non-Jordan
 because ${\rm Alt}_n$ is simple
for $n\geqslant 5$ and $|{\rm Alt}_n|=n!/2\xrightarrow[n\to\infty]{\ }\infty$.
\quad $\square$ \end{example}

Using
Example \ref{nJ1} one obtains a finitely generated non-Jordan group:

\begin{example}\label{nJ2} Let $\mathcal N$ be the subgroup of ${\rm Sym}_\infty$ generated by the transposi\-ti\-on $\sigma:=(1,2)$ and the ``translation'' $\delta$ defined by the condition
\begin{equation*}
\delta(i)=i+1\quad \mbox{for every $i\in \mathbf Z$}.
\end{equation*}
Then $\delta^m\sigma\delta^{-m}$ is the transposition $(m+1, m+2)$ for every $m$.
Since the set of transpositions $(1,2), (2,3),\ldots, (n-1, n)$ generates the symmetric group ${\rm Sym}_n$, this shows that $\mathcal N$ contains ${\rm Alt}_n$ for every $n$; whence $\mathcal N$ is non-Jordan.
\quad $\square$ \end{example}

One can show that $\mathcal N$ is not finitely presented. Here is an example of a finitely presented  non-Jordan group which is also simple.


\begin{example} Consider Richard J. Thompson's group $V$, see \cite[\S{6}]{CFP}.
It is finitely presented, simple and
contains a subgroup isomorphic with
${\rm Sym}_n$ for every $n\geqslant 2$.
The latter implies,
as  in
Example \ref{nJ1}, that
$V$ is non-Jordan.
(I thank Vic.\;Kulikov who drew my attention to this example.) \quad $\square$
\end{example}

\subsection{General properties}

\subsubsection{Subgroups, quotient groups, and products}
Exploring whether a group is Jordan or not leads to
the questions on
the connections between
Jordaness of a group, its subgroup, and its quotient group.

\begin{theorem}[{{\rm \cite[Lemmas 2.6, 2.7, 2.8]{Po1}}}] \label{gs}
\

\begin{enumerate}[\hskip 4.2mm \rm (1)]
\item Let $H$ be a subgroup of a group $G$.
\begin{enumerate}[\hskip -.9mm \rm (i)]
\item If $\,G$ is Jordan, then $H$ is Jordan and $J_H\leqslant J_G$.
\item If $\,G$ is Jordan and $H$ is normal in $\,G$, then
$G/H$ is Jordan and $J_{G/H}\leqslant J_G$ in either of the cases:
\begin{enumerate}[\hskip -.4mm \rm (a)]
\item $H$ is finite;
\item the extension $1\to H\to G\to G/H\to 1$ splits.
\end{enumerate}
\item If  $H$ is torsion-free, normal in $G$, and $\,G/H$ is Jordan,  then $\,G$ is Jordan and
$J_{G}\leqslant J_{G\!/\!H}$.
\end{enumerate}
\item Let $G_1$ and $G_2$ be two groups.\;Then $G_1\times G_2$ is Jordan if and only if $G_1$ and $G_2$ are.\;In this case, $J_{G_i}\leqslant J_{G_1\times G_2}\leqslant J_{G_1}J_{G_2}$ for every $i$.
\end{enumerate}
\end{theorem}
\begin{proof} (1)(i). This follows from Definition \ref{Jd}.

\vskip .4mm

If $H$ is normal in $G$, let $\pi\colon G\to G/H$ be the natural projection.
\vskip .4mm

(1)(ii)(a).  Let $F$ be a finite subgroup of $G/H$. Since $H$ is finite,
$\pi^{-1}(F)$ is finite. Since $G$ is Jordan, $\pi^{-1}(F)$ contains a normal abelian subgroup $A$ whose index is at most $J^{}_G$. Hence $\pi(A)$ is a normal abelian subgroup of $F$ whose index in $F$ is at most\;$J_G$.

\vskip .4mm

(1)(ii)(b). By the condition, there is a subgroup $S$ in $G$ such that
$\pi|_S: S\to G/H$ is an isomorphism; whence the claim by (1)(i).

\vskip .4mm

(1)(iii).
 Let $F$ be a finite subgroup of $G$. Since $H$ is torsion free, $F\cap H=\{1\}$; whence $\pi|_F\colon S\to \pi(F)$ is an isomorphism.\;Therefore, as $G/H$ is Jordan, $F$ contains a normal abelian subgroup whose index in $F$ is at most $J_{G\!/\!H}$.

\vskip .4mm

(2) If $G:=G_1\times G_2$ is Jordan, then (1)(i) implies that $G_1$ and $G_2$ are Jordan and $J_{G_i}\leqslant J_{G}$ for every $i$. Conversely, let $G_1$ and $G_2$ be Jordan. Let $\pi^{}_i\colon G\to G_i$ be the natural projection. Take a finite subgroup $F$ of $G$. Then $F^{}_i:=\pi^{}_i(F)$ contains an abelian normal subgroup $A^{}_i$ such that
\begin{equation}\label{ne}
[F^{}_i:A^{}_i]\leqslant J^{}_{G_i}.
\end{equation}
The subgroup ${\widetilde A}^{}_i:=\pi^{-1}_i(A^{}_i)\cap F$ is  normal in $F$ and $F/{\widetilde A}^{}_i$ is isomorphic
to $F^{}_i/A^{}_i$. From \eqref{ne} we then conclude that
\begin{equation}\label{nee}
[F:{\widetilde A}^{}_i]\leqslant J^{}_{G_i}.
\end{equation}

Since $A:={\widetilde A}^{}_1 \cap {\widetilde A}^{}_2$ is the kernel of the diagonal homomorphism
$$
F\longrightarrow
F/{\widetilde A}^{}_1\times F/{\widetilde A}^{}_2
$$
determined by the canonical projection $F\to F/{\widetilde A}^{}_i$,  we infer from \eqref{nee} that
\begin{equation}\label{neee}
[F:A]=|F/A|\leqslant |
F/{\widetilde A}^{}_1\times F/{\widetilde A}^{}_2|=|F_1/A_1||F_2/A_2|\leqslant J^{}_{G_1}J^{}_{G_2} .
\end{equation}
By construction, $A\subseteq A_1\times A_2$ and $A_i$ is abelian. Hence $A$ is abelian as well. Since $A$ is normal in $F$, the claim then follows from \eqref{neee}.
\quad $\square$ \end{proof}

\begin{theorem}\label{simple}
Let $H$ be a normal subgroup of a group $G$. If $H$ and $G/H$ are Jordan,
then any set of pairwise nonisomorphic simple nonabelian finite subgroups of $G$ is finite.
\end{theorem}

\begin{proof} Since up to isomorphism there are only finitely many finite groups of a fixed order,
Definition \ref{Jd} implies that any set of
pairwise nonisomorphic simple nonabelian finite subgroups of
a given Jordan group is finite. This implies the claim because simplicity of a finite subgroup $S$ of $G$ yields that either $S\subseteq H$ or
the canonical projection $G\to G/H$ embeds $S$ in $G/H$.
\quad $\square$\end{proof}

\subsubsection{Counterexample}

For a normal subgroup $H$ of $G$, it is not true, in general, that $G$ is Jordan if $H$ and $G/H$ are.
\begin{example}\label{Zar}
For every integer $n\!>\!0$ fix a finite group $G_n$ with the properties:
\begin{enumerate}[\hskip 4.2mm \rm (i)]
\item $G_n$ has an abelian normal subgroup $H_n$ such that $G_n/H_n$ is abelian;
\item there is a subgroup $Q_n$ of  $G_n$ such that the index in $Q_n$ of every abelian subgroup of $Q_n$
is greater or equal than $n$.
\end{enumerate}
Such a $G_n$ exists, see below. Now take
$G:={
\prod}_{n}G_n$ and $H:=\prod_n H_n$.\;Then
$H$ and $G/H$ are abelian by (i), hence Jordan, but $G$ is not Jordan by (ii).

The following construction from \cite[Sect.\,3]{Z} proves the existence of such a $G_n$.\;Let $K$ be a finite commutative group of order $n$ written additively and let $\widehat K:={\rm Hom}(K, k^\ast)$ be the group of characters of $K$ written multiplicatively. The formula
\begin{equation}\label{grla}
(\alpha, g, \ell)(\alpha', g', \ell'):= (\alpha\alpha'\ell'(g), g+g', \ell\ell')
\end{equation}
endows the set $k^\ast\times K\times \widehat K$ with the group structure.\;Denote by $G_K$ the obtained group. It is embedded in the exact sequence of groups
\begin{equation*}
\begin{gathered}
\{1\}\to k^\ast\xrightarrow{\iota} G_K\xrightarrow{\pi}K\times\widehat K\to \{(0,1)\},\\[-.5mm]
\mbox{where\;\;
$\iota(\alpha):=(\alpha, 0, 1)$\;\; and \;\;$\pi((\alpha, g, \ell)):=(g, \ell)$.}
\end{gathered}
\end{equation*}
Thus, if one takes $G_n:=G_K$ and $H_n:=\iota(k^\ast)$, then property (i) holds.\;Let $\mu_n$ be the subgroup of all $n$th roots of unity in $k^\ast$.\;From \eqref{grla} and $|K|=n$ we infer that the subset $Q_K:=\mu_n\times K\times \widehat K$ is a subgroup of $G_K$.\;In \cite[Sect.\,3]{Z} is proved that for $Q_n=Q_K$ property (ii) holds.
\quad $\square$ \end{example}

\subsubsection{Bounded groups}

However, under certain conditions, $G$ is Jordan if and only if $H$ and $G/H$ are.\;An example
of such a condition is given in Theorem \ref{JJJJ} below; it is based on Definition \ref{b} below introduced in \cite{Po1}.

Given a group $G$, put
\begin{equation*}
b_G:=\underset{F}{\rm sup}\,|F|,
\end{equation*}
where $F$ runs over all finite subgroups of $G$.

\begin{definition}[{{\rm \cite[Def.\;2.9]{Po1}}}]\label{b} A group $G$ is called {\it bounded} if
$b_G\neq \infty$.
\end{definition}

\begin{example} \label{bbbb}
Finite groups and torsion free groups are bounded.
\quad $\square$ \end{example}

\begin{example}\label{ebb}
 It is immediate from Definition \ref{b} that every extension of a bounded group by bounded is bounded.
\quad $\square$ \end{example}

\begin{example}\label{bou} By the classical Minkowski's theorem
${\bf GL}_n({\bf Z})$
is bounded (see,\;e.g.,
\cite[Thm.\;39.4]{Hu}).\;Since every finite subgroup of ${\bf GL}_n({\bf Q})$ is conjugate to a subgroup of ${\bf GL}_n(\bf Z)$ (see,\;e.g.,\,\cite[Thm.\;73.5]{CR}), this implies that
${\bf GL}_n({\bf Q})$ is bounded and
$b_{{\bf GL}_n(\bf Q)}=b_{{\bf GL}_n(\bf Z)}$.\;H.\;Minkowski
 and I.\;Schur obtained the following upper bound for $b_{{\bf GL}_n(\bf Z)}$,
see, e.g., \cite[\S39]{Hu}.\;Let ${\mathcal P}(n)$ be the set of all primes $p\in \bf N$ such that $[n/(p-1)]>0$. Then
\begin{equation}\label{Mi}
\qquad b_{{\bf GL}_n(\bf Z)}\leqslant \prod_{p\in {\mathcal P}(n)}\hskip -1.5mm p^{d_p}, \quad\mbox{where}\quad
d_p=\sum_{i=0}^{\infty}\biggl[\frac{n}{p^i(p-1)}\biggr].
\end{equation}
In particular, the right-hand side of the inequality in \eqref{Mi} is
\begin{equation*}
2, 24, 48, 5760, 11520, 2903040\quad\mbox{resp., for $n=1, 2, 3, 4, 5, 6$}.\qquad \square
\end{equation*}
\end{example}

\begin{example} Maintain the notation and assumption of Subsection \ref{manif}.\;If $\chi(M)\neq 0$, then by
\cite[Thm.\;1.4(2)]{MR}, the group ${\rm Diff}(M)$ is bounded.
Further information on smooth manifolds
with bounded diffeomorphism groups is contained in \cite{Pu}.
\quad $\square$\end{example}

\begin{example} Every bounded group $G$ is Jordan with $J_G\leqslant b_G$, and there are non-bounded Jordan groups (e.g., ${\bf GL}_n(k)$).
\quad $\square$ \end{example}

\begin{theorem}[{{\rm \cite[Lemma\;2.11]{Po1}}}]\label{JJJJ} Let $\,H$ be a normal subgroup of a
group \,$G$ such that $\,G/H$ is bounded.\;Then $\,G$ is Jordan if and only if $\,H$ is Jordan, and in this case $$J_G\leqslant b_{G/H}J_{H}^{b_{G/H}}.$$
\end{theorem}

\begin{proof}
A proof is needed only for the sufficiency.\;So let $H$ be Jordan and let
$F$ be a finite subgroup of $G$. By Definition \ref{Jd}
\begin{equation}\label{KH}
L:=F\cap H
\end{equation}
contains a normal abelian subgroup $A$
such that
\begin{equation}\label{L}
[L:A]\leqslant J_H.
\end{equation}
Let $g$ be an element of $F$. Since $L$ is a normal subgroup of $F$, we infer that
 $gAg^{-1}$
 is a normal abelian subgroup of $L$ and
 \begin{equation}\label{LL}
[L:A]=[L:gAg^{-1}].
\end{equation}

The abelian subgroup
\begin{equation}\label{M}
M:=\bigcap_{g\in F} gAg^{-1}.
\end{equation}
 is normal in $F$.\;We intend to prove
 that $[F:M]$ is upper bounded by a constant not depending on $F$. To 
 this end, fix  the representatives
$g_1,\ldots, g^{\ }_{|F/L|}$ of all cosets of $L$ in $F$. Then \eqref{M} and normality of $A$ in $L$ imply that
\begin{equation}\label{MM}
M=\bigcap_{i=1}^{|F/L|} g_iAg_i^{-1}.
\end{equation}
 From \eqref{MM} we deduce that  $M$ is the kernel of the diagonal homomorphism
 \begin{equation*}
 L\longrightarrow \prod_{i=1}^{|F/L|} L/g_iAg_i^{-1}
  \end{equation*}
  determined by the canonical projections $L\to L/g_iAg_i^{-1}$. This, \eqref{LL}, and \eqref{L} yield
  \begin{equation}\label{KLJ}
  [L:M]\leqslant [L:A]^{|F/L|}\leqslant J_{H}^{|F/L|}.
  \end{equation}

  Let $\pi\colon G\to G/H$ be the canonical projection.
  By  \eqref{KH} the finite subgroup
  $\pi(F)$ of
$G/H$ is isomorphic to $F/L$. Since $G/H$ is bounded, this yields
$|F/L|\leqslant b^{}_{G/H}$.
We then deduce from
 \eqref{KLJ} and $[F:M]=[F:L][L:M]$ that
 \begin{equation*}
 [F:M]\leqslant  b^{}_{G/H} J_{H}^{\,b^{}_{G/H}};
 \end{equation*}
whence the claim.
\quad $\square$ \end{proof}


The following corollary
should be compared with statement (1)(ii)(a) of Theorem\;\ref{gs}:

\begin{corollary}
Let $\,H$ be a finite normal subgroup of a group $\,G$ such that the center of $\,H$ is trivial. If $\,G/H$ is Jordan, then $\,G$ is Jordan and
\begin{equation*}
J_G\leqslant|{\rm Aut}(H)|J_{G/H}^{|{\rm Aut}(H)|}.
\end{equation*}
\end{corollary}
\begin{proof} Let $\varphi\colon G\to {\rm Aut}(H)$ be the homomorphism determined by the conjugating action of $G$ on $H$.\;Triviality of the center of $H$ yields
$H\cap {\rm ker}\,\varphi=\{1\}$.
Hence the restriction  of the natural projection $G\to G/H$ to ${\rm ker}\,\varphi$
is an embedding ${\rm ker}\,\varphi \hookrightarrow G/H$.\;Therefore,
${\rm ker}\,\varphi$ is Jordan since $G/H$ is.\;But $G/{\rm ker}\,\varphi$ is finite since it is isomorphic to a subgroup of ${\rm Aut}(H)$ for the finite group $H$.
By Theorem \ref{JJJJ} this implies the claim.
\quad $\square$ \end{proof}

\section{When are \boldmath ${\rm Aut}(X)$ and ${\rm Bir}(X)$
Jordan?}

\subsection{Problems A and B}\label{se}
In \cite[Sect.\,2]{Po1} were posed the
following two
problems:
\vskip 2mm
\noindent
{\bf Problem A.} Describe algebraic varieties $X$ for which
${\rm Aut}(X)$ is Jordan.

\noindent
{\bf Problem B.}
The same with ${\rm Aut}(X)$ replaced by ${\rm Bir}(X)$.

 \vskip 2mm

    Note that for rational varieties
    $X$
    Problem B means finding  $n$ such that the Cremona group ${\rm Cr}_n$ is Jordan;
    in this case, it was essentially  posed in \cite[6.1]{Se1}.

 Describing
 finite subgroups of the groups ${\rm Aut}(X)$ and ${\rm Bir}(X)$ for various varieties $X$
is a classical research direction, currently flourishing.\;Understan\-ding
 which of these groups are Jordan sheds
 a light on the structure of these subgroups.
 Varieties $X$ with non-Jordan group ${\rm Bir}(X)$ or
 ${\rm Aut}(X)$ are, in a sense, more ``symmetric'' and, therefore, more remarkable
 than those with Jordan group.
 The discussion below supports the conclusion that they occur ``rarely'' and their finding is a challenge.

\subsection{Groups \boldmath ${\rm Aut}(X)$}
In this subsection we shall consider
Problem A.

\begin{lemma}\label{comp} Let $X_1,\ldots , X_n$ be all the irreducible
components of a variety $X$. If every ${\rm Aut}(X_i)$ is Jordan, then
${\rm Aut}(X)$ is Jordan.
\end{lemma}
\begin{proof} Define the homomorphism
$\pi\colon {\rm Aut}(X)\to {\rm Sym}_n$ by $g\cdot X_i=X_{\pi(g)}$ for $g\in {\rm Aut}(X)$.\;Then $g\cdot X_i=X_i$ for every $g\in {\rm Ker}(\pi)$ and $i$, so the homomorphism $\pi_i\colon {\rm Ker}(\pi)\to {\rm Aut}(X_i)$, $g\mapsto g|_{X_i}$, arises.\;The definition implies that $\pi_1\times\ldots\times \pi_n\colon {\rm Ker}(\pi)\to \prod_{i=1}^n {\rm Aut}(X_i)$ is an injection; whence ${\rm Ker}(\pi)$ is Jordan
by Theorem \ref{gs}(2). Hence ${\rm Aut}(X)$ is Jordan by Theorem\;\ref{JJJJ}.
\quad $\square$ \end{proof}

At this writing (October 2013), not a single variety $X$ with non-Jor\-dan ${\rm Aut}(X)$  is known (to me).
\begin{question}[{{\rm \cite[Quest.\;2.30 and 2.14]{Po1}}}] \label{Jord} Is there an irreducible
variety $X$ such that
${\rm Aut}(X)$ is non-Jordan{\rm?}
Is there an irreducible affine variety $X$ with this property?
\end{question}

\begin{remark} One may consider the counterpart of the first question replacing $X$ by
a connected smooth topological manifold $M$, and ${\rm Aut}(X)$ by ${\rm Diff}(M)$.
The following yields the affirmative answer:

\begin{theorem}[\rm \cite{Po5}]\label{Man}
There is a simply connected noncompact smooth oriented $4$-dimensio\-nal manifold $M$ such that  ${\rm Diff}(M)$ contains
an isomorphic copy of every finitely presented $($in particular, of every finite\hskip .3mm$)$ group. This copy acts on $M$ properly doscontinuously.
\end{theorem}

Clearly,  ${\rm Diff}(M)$ is non-Jordan. By \cite[Thm.\,2]{Po5}, ``noncompact'' in Theorem \ref{Man}
cannot be replaced by ``compact''.
The following question (I reformulate it using De\-fi\-nition \ref{Jd}) was posed by \'E.\;Ghys (see \cite[Quest. 13.1]{Fi}): Is the diffeomorphism group of any compact smooth manifold Jordan?
In fact, according to \cite{MR13}, \'E.\;Ghys conjectured the affirmative answer.

\end{remark}

On the other hand, in
many cases it can be proven that ${\rm Aut}(X)$ is Jordan. Below are described several extensive classes of $X$ with this property.

\subsubsection{Toral varieties}

First, consider the
wide
class of affine  varieties singled out by the following

\begin{definition}[{{\rm \cite[Def.\,1.13]{Po1}}}] \label{def_toral}
A variety is called {\it toral\,} if it is isomorphic to a closed subvariety of some
${\mathbf A}^n\setminus \bigcup_{i=1}^n H_i$, where $H_i$ is the set of zeros of the $i$th standard coordinate function
$x_i$ on ${\mathbf A}^n$.
\end{definition}

\begin{remark}
${\mathbf A}^n\setminus \bigcup_{i=1}^n H_i$ is the group variety of
the $n$-dimensional affine torus; whence the
terminology.\;Warning:\;``toral'' does not imply ``affine to\-ric'' in the sense of
\cite{Fu}.
\end{remark}

The class of toral varieties is closed with respect to taking products and closed subvarieties.

\begin{lemma}[{{\rm \cite[Lemma 1.14(a)]{Po1}}}]\label{toral}
The following properties of an affine variety $X$ are equivalent:
\begin{enumerate}[\hskip 2.2mm \rm(i)]
\item $X$ is toral;
\item $k[X]$ is generated by $k[X]^*$,
the group  of units of $k[X]$.
\end{enumerate}
\end{lemma}
\begin{proof} If $X$ is closed in ${\mathbf A}^n\setminus \bigcup_{i=1}^n H_i$, then the restriction of functions is an epimorphism
$k\big[{\mathbf A}^n\setminus \bigcup_{i=1}^n H_i\big]\to k[X].$ Since
$k\big[{\mathbf A}^n\setminus \bigcup_{i=1}^n H_i\big]=k[x_1,\ldots, x_n, 1/x_1,\ldots 1/x_n]$, this proves ${\rm (i)}\Rightarrow {\rm (ii)}$.

Conversely, assume that $\rm (ii)$ holds and let
\begin{equation}\label{ca}
k[X]=k[f_1,\ldots, f_n]
\end{equation}
 for some $f_1,\ldots, f_n\!\in\! k[X]^*$.\;Since\;$X$\;is affine, \eqref{ca} implies that $\iota \colon X\to {\bf A}\!^n$,
$x\mapsto (f_1(x), \ldots, f_n(x))$, is a closed embedding.\;The
standard coordinate functions on ${\bf A}\!^n$ do not vanish
on $\iota(X)$ since every $f_i$ does not vanish on $X$.\;Hence $\iota(X)\subseteq {\mathbf A}^n\setminus \bigcup_{i=1}^n H_i$. This proves (ii)$\Rightarrow$(i). \quad $\square$ \end{proof}

\begin{lemma}\label{cov} Any quasiprojective variety $X$ endowed with
a finite automorphism group $G$ is
covered by
$\,G$-stable toral open subsets.
\end{lemma}
\begin{proof}  First, any point $x\in X$ is contained in a $G$-stable affine open subset of $X$.\;Indeed, since the orbit $G\cdot x$ is finite and $X$ is quasiprojective, there is an affine open subset $U$ of $X$ containing $G\cdot x$. Hence
$V:=\bigcap_{g\in G} g\cdot U$
is a $G$-stable open subset containing $x$, and, since every
$g\cdot U$ is affine, $V$ is affine as well, see, e.g.,\;\cite[Prop.\,1.6.12(i)]{Sp}.

Thus, the problem is reduced to the case where $X$ is affine.\;Assume then that $X$ is affine, and let $k[X]=k[h_1,\ldots, h_s]$.\;Replacing $h_i$ by $h_i+\alpha_i$ for an appropriate $\alpha_i\in k$, we may (and shall) assume that every $h_i$ vanishes nowhere on the $G\cdot x$.\;Expanding the set $\{h_1,\ldots, h_s\}$ by including $g\cdot h_i$ for every $i$ and $g
\in G$, we may (and shall) assume that
$\{h_1,\ldots, h_s\}$ is $G$-stable.\;Then $h:=h_1\cdots h_s\in k[X]^G$. Hence the affine open set $X_h:=\{z\in X \mid h(z)\neq 0\}$ is $G$-stable and contains $G\cdot x$.
Since $k[X_h]=k[h_1,\ldots, h_s, 1/h]$ and $h_1,\ldots, h_s, 1/h\in k[X_h]^*$,\; the variety $X_h$ is toral by Lemma \ref{toral}.
\quad $\square$ \end{proof}

\begin{remark} Lemma \ref{cov}
and its proof
remain true for any variety $X$ such that every $G$-orbit
is contained in an affine open subset; whence the following
\end{remark}

\begin{corollary} Every
variety is
covered by
open toral subsets.
\end{corollary}

For irreducible toral varieties the following was proved   in \cite[Thm. 2.16]{Po1}.

\begin{theorem} \label{jtoral} The automorphism group of every toral variety is Jordan.
\end{theorem}

\begin{proof} By Theorem \ref{comp} it suffices to prove this for irreducible toral varieties.

By \cite{Ro}, for any irreducible variety $X$,
 $$\Gamma:=k[X]^*/k^*$$ is
a  free abelian group of
finite rank.\;Let $X$ be toral and let
$H$ be the kernel of
the natural action of ${\rm Aut}(X)$ on $\Gamma$.\;We claim that $H$ is abelian.
Indeed, for every function $f\in k[X]^*$, the line in $k[X]$ spanned over $k$ by $f$
is $H$-stable. Since ${\bf GL}_1$ is abelian,
this yields that
\begin{equation}\label{hh}
h_1h_2\!\cdot\! f=h_2h_1\!\cdot\! f\quad \mbox{for any elements\quad $h_1, h_2\in H$}.
 \end{equation}
   As $X$ is toral, $k[X]^*$ generates the $k$-algebra $k[X]$ by Lemma
\ref{toral}. Hence \eqref{hh} holds for every $f\in k[X]$. Since $X$ is affine, the automorphisms of $X$ coincide if and only if they induce the same automorphisms of  $k[X]$.\;Therefore, $H$ is abelian, as claimed.

Let $n$ be the rank of $\Gamma$.\;Then ${\rm Aut}(\Gamma)$ is isomorphic to ${\bf GL}_n({\bf Z})$.\;By the definition of $H$, the natural action of
${\rm Aut}(X)$ on $\Gamma$ induces an embedding of ${\rm Aut}(X)/H$ into ${\rm Aut}(\Gamma)$.~Hence ${\rm Aut}(X)/H$ is isomorphic to a subgroup of ${\bf GL}_n({\bf Z})$ and therefore
is bounded by
Example \ref{bbbb}(2).\;Thus,
${\rm Aut}(X)$ is an extension of a bounded group by an abelian group, hence Jordan by Theorem\;\ref{JJJJ}.
This completes the proof.
\quad $\square$ \end{proof}

\begin{remark} {\rm
Maintain the notation of the proof of Theorem \ref{jtoral} and assume that $X$ is irreducible.\;Let $f_1,\ldots f_n$ be a basis of $\Gamma$.\;There are the homomorphisms $\lambda_i\colon H\to k^*$,  $i=1,\ldots, n$, such that $h\cdot f_i=\lambda(h)f_i$ for every $h\in H$ and $i$.\;Since $k[X]^*$ generates $k[X]$, the diagonal map
$
H\to (k^*)^n,\hskip 1mm h\mapsto (\lambda_1(h),\ldots,\lambda_n(h)),
$
is injective. This and the proof of Theorem \ref{jtoral} show that for any irreducible toral variety $X$
with ${\rm rk}\,k[X]^*/k^*=n$, there is an exact sequence
\begin{equation*}
\{1\}\to D\to {\rm Aut}(X)\to B\to \{1\},
\end{equation*}
where $D$ is a subgroup of the torus $(k^*)^n$ and $B$ is a subgroup of ${\bf GL}_n({\bf Z})$.
}
\end{remark}

Combining Theorem \ref{jtoral}  with Corollary of Lemma \ref{cov},  we get the following:

\begin{theorem}\label{neighb}
Any point of any variety has an open neighborhood $U$ such that ${\rm Aut}(U)$ is Jordan.
\end{theorem}

\subsubsection{Affine spaces}\label{as}

Next, consider the fundamental objects of algebraic geometry, the affine spaces ${\mathbf A}^{\hskip -.3mm n}$.\;The group ${\rm Aut}({\mathbf A}^{\hskip -.3mm n})$
is the ``{\it af\-fine Cremona group of rank} $n$''.

Since ${\rm Aut}({\mathbf A}^{\hskip -.5mm 1})$ is the affine algebraic group ${\rm Aff}\!_1$, it is Jordan by Theorem\;\ref{lg}.

Since ${\rm Aut}({\mathbf A}^{\hskip -.3mm 2})$ is  the subgroup of ${\rm Cr}_2$, it is Jordan by Example \ref{Cr2}.\;Another proof: By \cite{Ig} every finite subgroup of ${\rm Aut}({\mathbf A}^{\hskip -.3mm 2})$ is conjugate to a subgroup of ${\bf GL}_2(k)$, so the claim follows from Theorem \ref{Jt}.

The group ${\rm Aut}({\mathbf A}^{\hskip -.3mm 3})$ is Jordan being the subgroup of
${\rm Cr}_3$ that is Jordan by Corollary \ref{Cr3} below.

 At this writing (October 2013) is unknown whether ${\rm Aut}({\mathbf A}^{\hskip -.3mm n})$  is Jordan
for $n\geqslant 4$ or not.\;By Theorem \ref{RC}, if the so-called BAB Conjecture (see Subsection \ref{recent}    below) holds true in dimension $n$,
then
${\rm Cr}_n$
is Jordan, hence
${\rm Aut}({\mathbf A}^{\hskip -.3mm n})$ is Jordan as well.

\subsubsection{Fixed points and Jordaness}\label{fpm}

The following method of proving Jordaness of  ${\rm Aut}(X)$ was
suggested in \cite[Sect.\;2]{Po1} and  provides extensive classes of
$X$ with Jordan ${\rm Aut}(X)$.\;It is based on the use of
the following
fact:

\begin{lemma} \label{linea}
Let $\,X$ be an irreducible variety, let $\,G$ be a finite subgroup of ${\rm Aut}(X)$, and let $\,x\in X$ be a fixed point of $\,G$.\;Then the natural action of $\,G$ on ${\rm T}_{x, X}$, the tangent space of $X$ at $x$,  is faithful.
\end{lemma}
\begin{proof}
Let
$\mathfrak m_{x, X}$ be
the maximal ideal of
$\mathcal O_{x, X}$,
the local ring of $X$ at $x$.\;Being finite, $G$ is reductive.\;Since ${\rm char}\,k=0$, this implies that $\mathfrak m_{x, X}=L\oplus \mathfrak m_{x, X}^2$ for some submodule $L$ of the $G$-module
$\mathfrak m_{x, X}$.\;Let $K$ be the kernel of the action of $G$ on $L$ and let $L^d$ be the $k$-linear span in $\mathfrak m_{x, X}$ of the $d$th powers of all the elements of $L$. By the Nakayama's Lemma, the
restriction to $L^d$ of the natural projection $\mathfrak m_{x, X}\to \mathfrak m_{x, X}/\mathfrak m_{x, X}^{d+1}$ is surjective. Hence $K$ acts trivially on $\mathfrak m_{x, X}/\mathfrak m_{x, X}^{d+1}$ for every $d$.

 Take an element $f\in \mathfrak m_{x, X}$.\;Since $G$ is finite, the $k$-linear span $\langle K\cdot f\rangle$ of the $K$-orbit of $f$ in $\mathfrak m_{x, X}$ is finite-dimensional.\;This and $\bigcap_{s}\mathfrak m_{x, X}^s=\{0\}$
(see,\;e.g.,\;\cite[Cor.\;10.18]{AM}) implies that  $\langle K\cdot f\rangle \cap
\mathfrak m_{x, X}^{d+1}=\{0\}$ for some $d$.  Since $f-g\cdot f\in \mathfrak m_{x, X}^{d+1}$ for every element $g\in K$, we conclude that $f=g\cdot f$, i.e., $f$ is $K$-invariant. Thus, $K$ acts trivially on $\mathfrak m_{x, X}$, hence on $\mathcal O_{x, X}$ as well. Since $k(X)$ is the field of fractions of $\mathcal O_{x, X}$,
$K$ acts trivially on $k(X)$, and therefore, on $X$. But $K$ acts on $X$ faithfully because $K\subseteq {\rm Aut}(X)$.\;This proves that $K$ is trivial.\;Since $L$ is the dual of the $G$-module ${\rm T}_{x, X}$, this completes the proof.
\quad $\square$ \end{proof}

The idea of the method is to use the fact that
 if a finite subgroup $G$ of ${\rm Aut}(X)$
has a fixed point $x\in X$, then, by Lemma \ref{linea} and Theorem \ref{Jt}, there is a normal abelian subgroup
of $G$ whose index in $\,G$ is at most $J_{{\bf GL}_n(k)}$ for  $n=\dim {\rm T}_{x, X}$.

This yields the following:

\begin{theorem}\label{aaa} Let $X$ be an irreducible variety and let
$\,G$ be a finite subgroup of ${\rm Aut}(X)$.\;If $G$ has a fixed point in $X$,
then there is a normal abelian subgroup of $\,G$
whose index in $\,G$ is at most $J_{{\bf GL}_m(k)}$, where
\begin{equation}\label{m}
m=\underset{x}{\max} \dim {\rm T}_{x, X}.
\end{equation}
\end{theorem}

\begin{corollary}
\label{fp}
If every finite automorphism group of an irreducible variety $X$ has a fixed point in $X$, then  ${\rm Aut}(X)$ is Jordan and
\begin{equation*}
J_{{\rm Aut}(X)}\leqslant J_{{\bf GL}_{m}(k)},
\end{equation*}
where $m$ is defined by \eqref{m}.
\end{corollary}

\begin{corollary}\label{ccccccc} Let $p$ be a prime number.\;Then every finite
$p$-subgroup $\,G$ of
${\rm Aut}({\bf A}\!^n)$ contains an abelian normal subgroup
whose index in $\,G$ is at most $J_{{\bf GL}_n(k)}$.
\end{corollary}

\begin{proof}
 This follows from Theorem \ref{aaa} since in this case $({\bf A}\!^n)^G\neq \varnothing$, see \cite[Thm.\;1.2]{Se3}.
\quad $\square$ \end{proof}

\begin{remark}{\rm At this writing (October 2013), it is unknown whether or not
$({\bf A}\!^{n})^G\neq\varnothing$ for every
finite subgroup $\,G$ of ${\rm Aut}({\bf A}\!^{n})$.\;By Theorem
\ref{aaa} the affirmative answer would imply that ${\rm Aut}({\bf A}\!^{n})$ is Jordan
(cf.\;Subsection \eqref{as}).
}
\end{remark}

\begin{remark} {\rm The statement
of Corollary \ref{ccccccc}
remains true if ${\bf A}\!^n$ is replaced by any $p$-acyclic variety $X$,
and $n$ in $J_{{\bf GL}_n(k)}$ is replaced by   $m$ (see \eqref{m}).
This is because
in this case
 $X^G\neq \varnothing$ for every finite $p$-subgroup $G$ of ${\rm Aut}(X)$,
see \cite[Sect.\;7--8]{Se3}.}
\end{remark}

The following applications are obtained by combining the above idea with Theorem\;\ref{JJJJ}.

\begin{theorem}\label{equivvv}
Let $X$ be an irreducible variety.\;Consider an ${\rm Aut}(X)$-stable
equiva\-lence relation $\sim$
on the set its points.\;If there is a finite equivalence class $C$ of $\sim$, then
${\rm Aut}(X)$ is Jordan and
\begin{equation*}
J_{{\rm Aut}(X)}\leqslant |C|! J_{{\bf GL}_m(k)}^{|C|!},
\end{equation*}
where $m$ is defined by {\rm\eqref{m}}.
\end{theorem}
\begin{proof} By the assumption, every equivalence class of $\sim$ is
 ${\rm Aut}(X)$-stable. The kernel $K$ of the action of ${\rm Aut}(X)$ on $C$ is a normal subgroup
 of ${\rm Aut}(X)$ and, since the elements of ${\rm Aut}(X)$ induce permutations of $C$,
 \begin{equation}\label{C!}
 [{\rm Aut}(X)\!:\!K]\leqslant |C|!.
 \end{equation}
  By Theorem \ref{JJJJ}, Jordaness of ${\rm Aut}(X)$
 follows from  that of $K$.\;To prove that the latter holds, take a point of $x\in C$.\;Since $x$ is fixed by every finite subgroup of $K$,
Theorem \ref{aaa} implies that $K$ is Jordan and $J_K\leqslant J_{{\bf GL}_m(k)}$.
By Theorem\;\ref{JJJJ}, this and \eqref{C!} imply the claim.
\quad $\square$ \end{proof}

\begin{example}\label{local}
Below are several examples of ${\rm Aut}(X)$-stable equivalence relations on
an irreducible variety $X$:
\begin{enumerate}[\hskip 4.2mm \rm(i)]

\item $x\sim y\!\!\iff\!\! \mbox{$\mathcal O_{x,X}$ and $\mathcal O_{y,X}$ are $k$-isomorphic}$;
\item $x\sim y\!\!  \iff\!\!  \dim {\rm T}_{x, X}=\dim {\rm T}_{y, X}$;
\item $x\sim y\!\!  \iff \!\! \mbox{the tangent cones of $\,X$ at $\,x$ and $y$ are isomorphic}$.\quad $\square$
\end{enumerate}
\end{example}

\begin{corollary}\label{lr}
If an irreducible variety $X$ has a point $x$ such that the set
\begin{equation*}
\{y\in X\mid \mbox{$\mathcal O_{x,X}$ and $\mathcal O_{y,X}$ are $k$-isomorphic}\}
 \end{equation*}
 is finite, then ${\rm Aut}(X)$ is Jordan.
\end{corollary}

Call a point $x\in X$ a {\it vertex} of $X$ if
\begin{equation*}
\dim {\rm T}_{x, X}\geqslant \dim {\rm T}_{y, X}\hskip 2mm\mbox{for every point}\hskip 2mm y\in X.
\end{equation*}
Thus every point of $X$ is a vertex of $X$ if and only if
$X$ is smooth.
\begin{corollary} \label{nonsmooth}
The automorphism group of every irreducible variety
with only finitely many vertices is Jordan.
\end{corollary}

\begin{corollary}\label{sssing}  The automorphism group of every nonsmooth irreducible
variety with only finitely many  singular points is Jordan.
\end{corollary}

\begin{corollary}\label{connn}
Let
$X\subset {\bf A}^{\hskip -.4mm n}$ be the affine cone of a smooth closed proper
irreducible subvariety $Z$ of $\,{\bf P}^{n-1}$
that does not lie in any hyperplane.
Then
${\rm Aut}(X)$ is Jordan.
\end{corollary}

\begin{proof} The assumptions imply that the singular locus of $X$ consists of a single point, the origin; whence the claim by Corollary \ref{sssing}.
\quad $\square$ \end{proof}

\begin{corollary}\label{coooone}
\noindent If an irreducible variety $X$ has a point $x$ such that there are only finitely many points
$y\in X$ for which the tangent cones of $X$ at $x$ and at $y$ are isomorphic,
then
${\rm Aut}(X)$ is Jordan.
\end{corollary}

\begin{remark}{\rm  Smoothness in Corollary \ref{connn}  may be replaced by the assumption that $Z$ is not a cone.\;Indeed, in this case the origin constitutes a single  equivalence class
of equivalence relation (iii) in Example \ref{local}; whence the claim by Corollary \ref{coooone}.}
\end{remark}

\subsubsection{The Koras--Russell threefolds} Let  $X=X_{d,s,l}$ be the so-called Koras--Russell threefold of the first kind \cite{M-J}, i.e.,\,the smooth hypersurface in ${\bf A}^{\hskip -.5mm 4}$ defined by the equation
\begin{equation*}
x_1^dx_2+x_3^s+x_4^l+x_1=0,
\end{equation*}
where $d\geqslant 2$ and $2\leqslant s\leqslant l$ with $s$ and $l$ relatively prime; the case $d=s=2$ and $l=3$ is 
the famous Koras--Russell cubic.\;According
 to \cite[Cor.\;6.1]{M-J}, every element of ${\rm Aut}(X)$ fixes the origin $(0,0,0,0)\in X$.\;By
 Corollary
 \ref{fp} and item (iii) of Subsection \ref{JJJJJ}
this implies that ${\rm Aut}(X)$ is Jordan and
\begin{equation*}
J^{}_{{\rm Aut}(X)}\leqslant 360.
\end{equation*}

Actually, during the conference I learned from L.\;Moser-Jauslin that $X$ contains a line $\ell$ passing through the origin, stable with respect to ${\rm Aut}(X)$, and such that every element of ${\rm Aut}(X)$
fixing $\ell$ pointwise has infinite order.\;This implies that every finite subgroup of  ${\rm Aut}(X)$ is cyclic
and hence $$J^{}_{{\rm Aut}(X)}=1.$$

\subsubsection{Small dimensions}\label{smalld}

Since ${\rm Aut}(X)$ is a subgroup of ${\rm Bir}(X)$, Jordaness of ${\rm Bir}(X)$ implies that of ${\rm Aut}(X)$. This and Theorem \ref{CS} below yield the following
\begin{theorem} Let $X$ be an irreducible variety of dimension $\leqslant 2$ not birationally isomorphic to ${\bf P}^1\times E$, where $E$ is an elliptic curve.\;Then ${\rm Aut}(X)$ is Jordan.
\end{theorem}

Note that if $E$ is an elliptic curve and $X={\bf P}^1\times E$, then
${\rm Aut}(X)={\bf PGL}_2(k)\times {\rm Aut}(E)$, see \cite[pp.\;98--99]{Mar}.\;Fixing a point of $E$, endow $E$ with a structure of abelian variety $E_{\rm ab}$.
Since ${\rm Aut}(E)$ is an extension of the finite group ${\rm Aut}(E_{\rm ab})$
by the abelian group $E_{\rm ab}$, Theorems \ref{lg}, \ref{JJJJ}, and \ref{gs}(2)
imply that
 ${\rm Aut}({\bf P}^1\times E)$ is Jordan.

 Note also that all irreducible curves (not necessarily smooth and projective) whose
automorphism group is infinite are classified in~\cite{Po3}.

\subsubsection{Non-uniruled varieties}

Again, using that Jordaness of ${\rm Bir}(X)$ implies that of
${\rm Aut}(X)$,  we deduce from recent Theorem \ref{uniru}(i)(a) below the following
\begin{theorem}
${\rm Aut}(X)$ is Jordan for any irreducible non-uniruled variety\;$X$.
\end{theorem}

\subsection{Groups \boldmath ${\rm Bir}(X)$}\label{Birrr} Now we shall consider
Problem B (see Subsection \ref{se}).\;Exploring  ${\rm Bir}(X)$, one may, maintaining this group,
replace $X$ by any variety birationally isomorphic to $X$.\;Note that
by Theorem \ref{neighb}
one can always attain that after such a replacement ${\rm Aut}(X)$ becomes Jordan.

The counterpart of Question \ref{Jord} is
\begin{question}[{{\rm \cite[Quest.\;2.31]{Po1}}}] \label{QB}
Is there an irreducible variety $X$
such that ${\rm Bir}(X)$ is non-Jordan{\rm ?}
\end{question}

In contrast to the case of Question\;\ref{Jord}, at present we know the answer to
Question\;\ref{QB}:
motivated by my question,
Yu.\;Zarhin proved in
\cite{Z} the following
\begin{theorem}[{{\rm \cite[Cor.\,1.3]{Z}}}]\label{nJ} Let $X$ be an abelian variety of positive dimension and let $Z$ be a rational variety of positive dimension.\;Then ${\rm Bir}(X\times Z)$ is non-Jordan.
\end{theorem}
\noindent{\it Sketch of proof.}
By Theorem \ref{gs}(1)(i),
it suffices to prove that ${\rm Bir}(X\!\times\! {\bf A}^{\hskip -.4mm 1})$ is non-Jordan.\;Consider an ample divisor $D$ on $X$ and the sheaf $L\!:=\!\mathcal O_X(D)$. For a positive integer $n$, consider  the
following group $\Theta(L^n)$.\;Its elements are all pairs $(x, [f])$ where $x\!\in\! X$ is such that $L^n \!\cong\! T_x^\ast (L^n)$ for the translation $T_x\colon X\!\to\! X$, $z\!\mapsto\! z+x$,
and $[f]$ is the automorphism of the additive group of $k(X)$ induced by the multiplication
by $f\!\in\! k(X)^\ast $.\;The group structure of $\Theta(L^n)$ is defined by $(x, [f])(y, [h])\!=\!(x+y, [T_x^\ast h\cdot f])$.\;One proves that $\Theta(L^n)$ enjoys the properties: (i)  $\varphi\colon
\Theta(L^n)\!\to\! {\rm Bir}(X\!\times\! {\bf A}^{\hskip -.4mm 1})$, $\varphi(x, [f])(y, t)\!=\!(x\!+\!y, f(y)t)$, is a group embedding; (ii) $\Theta(L^n)$ is isomorphic to a group $G_K$ from Example\;\ref{Zar} with $|K|\!\geqslant\! n$.\;This implies the claim (see Example\;\ref{Zar}).
\quad $\square$

\vskip 2mm

Below Problem B is solved for varieties of small dimensions ($\leqslant 2$).
 \subsubsection{Curves}
If $X$ is a curve, then the answer to Question \ref{QB}
is negative.

Proving this,
we may assume that
$X$ is smooth and projective; whence ${\rm Bir}(X)\!=\!{\rm Aut}(X)$.

If $g(X)$, the genus of $X$, is $0$,
then
$X={\bf P}^1$, so
${\rm Aut}(X)={\bf PGL}_2(k)$. Hence
${\rm Aut}(X)$ is Jordan by Theorem \ref{lg}.

If $g(X)=1$, then $X$ is an elliptic curve, hence
${\rm Aut}(X)$ is Jordan
 (see the penultimate paragraph in Subsection \ref{smalld}).

 If $g(X)\geqslant 2$, then, being finite, ${\rm Aut}(X)$ is Jordan.

\subsubsection{Surfaces}

Answering Question \ref{QB} for surfaces $X$, we may assume that $X$ is a smooth
projective minimal model.

If $X$ is of general type, then  by
Matsumura's
theorem
${\rm Bir}(X)$ is finite, hence Jordan.

If $X$ is rational, then ${\rm Bir}(X)$ is
${\rm Cr}_2$,
hence Jordan, see Example \ref{Cr2}.

If $X$ is a nonrational ruled surface, it is birationally isomorphic to ${\bf P}^1\times B$ where $B$ is a smooth projective curve
such that
$g(B)>0$; we may then take $X={\bf P}^1\times B$.\;Since
$g(B)>0$, there are no dominant rational maps ${\bf P}^1\dashrightarrow B$; whence
the elements of ${\rm Bir}(X)$ permute fibers of
the natural projection  ${\bf P}^1\times B\to B$.\;The set of ele\-ments
 inducing trivial permutation is a normal subgroup ${\rm Bir}^{}_B(X)$ of ${\rm Bir}(X)$.\;The definition implies that
 ${\rm Bir}^{}_B(X)={\bf PGL}_2(k(B)),
  $
  hence ${\rm Bir}^{}_B(X)$ is Jordan by Theorem \ref{lg}.\;Identifying ${\rm Aut}(B)$ with the subgroup of ${\rm Bir}(X)$ in the natural way, we get the decomposition
 \begin{equation}\label{semi}
 {\rm Bir}(X)={\rm Bir}^{}_B(X)\rtimes {\rm Aut}(B).
 \end{equation}
 If $g(B)\geqslant 2$,\;then ${\rm Aut}(B)$ is finite; whence
${\rm Bir}(X)$ is Jordan by virtue of \eqref{semi} and
Theorem\;\ref{JJJJ}.
If $g(B)\!=\!1$,
then  ${\rm Bir}(X)$ is non-Jordan by Theorem\;\ref{nJ}.

The canonical class of all the other surfaces $X$ is numerically effective, so, for them, ${\rm Bir}(X)={\rm Aut}(X)$, cf.\;\cite[Sect.\,7.1, Thm.\,1 and Sect.\;7.3, Thm.\;2]{IS}.

Let $X$ be such a surface.\;The group ${\rm Aut}(X)$ has a structure of a locally algebraic group with finite or countably many components,\;see \cite{Mat}, i.e.,  there is a normal subgroup ${\rm Aut}(X)^0$ in ${\rm Aut}(X)$ such that
\begin{enumerate}
\item[(i)]  ${\rm Aut}(X)^0$ is a connected algebraic group,
\item[(ii)]
${\rm Aut}(X)/{\rm Aut}(X)^0$ is either a finite or a countable group,
\end{enumerate}
By (i) and the structure theorem on algebraic groups \cite{Ba}, \cite{R1}
there is a normal connected affine algebraic subgroup $L$ of ${\rm Aut}(X)^0$ such that
${\rm Aut}(X)^0/L$ is an abelian variety.\;By \cite[Cor.\,1]{Matm} nontriviality of $L$ would imply that $X$ is ruled.\;Since we assumed that
$X$ is not ruled, this means that
${\rm Aut}(X)^0$ is an abelian variety.\;Hence ${\rm Aut}(X)^0$ is abelian and, a fortiori, Jordan.

By (i) the group ${\rm Aut}(X)^0$ is contained in the kernel of the natural action of ${\rm Aut}(X)$ on $H^2(X, {\bf Q})$ (we may assume that $k={\bf C})$.\;Therefore, this action defines a homomorphism  ${\rm Aut}(X)/{\rm Aut}(X)^0\to {\bf GL}(H^2(X, {\bf Q}))$.\;The kernel of this homomorphism is finite by
\cite[Prop.\,1]{Do}, and  the image is bounded by
Example \ref{bou}.\;By Examples \ref{bbbb}, \ref{ebb}  this yields that ${\rm Aut}(X)/{\rm Aut}(X)^0$ is bounded. In turn, since ${\rm Aut}(X)^0$ is Jordan, by Theorem \ref{JJJJ} this implies that ${\rm Aut}(X)$ is Jordan.

\vskip 2mm
\subsubsection{The upshot}


The upshot of the last two subsections is

 \begin{theorem}[{{\rm \cite[Thm.\;2.32]{Po1}}}]\label{CS} Let $X$ be an irreducible variety of dimension $\leqslant 2$.\;Then the following two properties are equivalent{\rm:}
 \begin{enumerate}
 \item[\rm(a)] the group ${\rm Bir}(X)$ is Jordan;
  \item[\rm(b)]  the variety $X$ is not birationally isomorphic to ${\bf P}^1\times B$, where $B$ is an elliptic curve.
      \end{enumerate}
\end{theorem}

\subsubsection{Finite and connected algebraic subgroups of $\,{\rm Bir}(X)$ and ${\rm Aut}(X)$}


Recall that the notions of
algebraic subgroup of $\,{\rm Bir}(X)$ and $\,{\rm Aut}(X)$ make sense, and every algebraic subgroup of $\,{\rm Aut}(X)$ is that of $\,{\rm Bir}(X)$, see,\;e.g., \cite[Sect.\,1]{P3}.\;Na\-me\-ly,
a map $\psi\colon S\to {\rm Bir}(X)$ of a variety $S$ is called an {\it algebraic family} if
the domain of definition of the partially defined map $\alpha\colon S\times X\to X$, $(s, x)\mapsto \psi(s)(x)$ contains a dense open subset of $S\times X$ and $\alpha$ coincides on it with a rational map $\varrho\colon S\times X\dashrightarrow X$.
The group $\,{\rm Bir}(X)$ is endowed with the {\it Zariski topo\-logy} \cite[Sect.\,1.6]{Se2}, in which a subset $Z$ of $\,{\rm Bir}(X)$ is closed if and only if $\psi^{-1}(Z)$ is closed in $S$ for every family $\psi$. If $S$ is an algebraic group and $\psi$ is an algebraic family which is a homomorphism of abstract groups, then
$\psi(S)$  is called an {\it algebraic subgroup of} $\;{\rm Bir}(X)$. In this case, ${\rm ker}(\psi)$ is closed in $S$
and the restriction to $\psi(S)$ of the Zariski topology of ${\rm Bir}(X)$ coincides with the topology
determined by the natural identification of $\psi(S)$ with the algebraic group $S/{\rm ker}(\psi)$.
If $\psi(S)\subset {\rm Aut}(X)$ and $\varrho$ is a morphism, then $\psi(S)$ is called an {\it algebraic subgroup of} $\,{\rm Aut}(X)$.

The following reveals a relation between embeddabi\-lity of finite subgroups
of ${\rm Bir}(X)$  in connected affine algebraic subgroups of
${\rm Bir}(X)$ and  Jordaness of ${\rm Bir}(X)$ (and the same holds for ${\rm Aut}(X)$).

For every integer $n\!>\!0$,
consider the   set of 	
 all isomorphism classes of con\-nected reductive algebraic groups of rank $\leqslant n$, and fix a group in
 every class. The obtained set of groups ${\mathcal R}_n$
is finite.\;Therefore,
\begin{equation}\label{Jn}
J_{\leqslant n}:=\underset{R\in {\mathcal R}_n}{\rm sup} J_{R}
\end{equation}
is a positive integer.

\begin{theorem}
\label{fcs}
Let $X$ be an irreducible variety of dimension $n$.\;Then every
finite subgroup $\,G$
of every connected affine
algebraic subgroup
 of $\,{\rm Bir}(X)$
has
a normal abelian subgroup whose index in $\,G$ is at most $J_{\leqslant n}$.
\end{theorem}

\begin{proof} Let $L$ be a connected affine algebraic subgroup of $\,{\rm Bir}(X)$ containing $G$.\;Being finite, $G$  is reductive.\;Let $R$ be a maximal reductive subgroup of $L$ containing $G$.\;Then
  $L$ is a semidirect product of $R$ and
  the unipotent radical of $L$, see\;\cite[Thm.\;7.1]{Mo}.\;Therefore, $R$ is connected because $L$ is.\;Faithfulness of the action $R$ acts on $X$ yields that
${\rm rk}\,R\leqslant \dim X$, see, e.g.,\,\cite[Lemma 2.4]{P3}.\;The claim then follows from \eqref{Jn}, Theorem \ref{lg}, and Defini\-tion\;\ref{Jd}.
\quad $\square$ \end{proof}

Theorem \ref{fcs} and Definition \ref{Jd} imply

\begin{corollary}\label{incl}
 Let $X$ be an irreducible variety of dimension $n$ such that
 ${\rm Bir}(X)$ \linebreak{\lb}resp.\;${\rm Aut}(X)\!)$ is non-Jordan.\;Then
 for every integer $d >\!J_{\leqslant n}$, there is a finite sub\-group $\,G$ of $\;{\rm Bir}(X)$
\lb resp.\;${\rm Aut}(X)\!)$  with the properties:
 \begin{enumerate}[\hskip 4.2mm\rm(i)]
\item $G$ does not lie in any connected affine algebraic subgroup of ${\rm Bir}(X)\!$
\linebreak \lb resp.\;${\rm Aut}(X)\!);$
\item
for any
abelian normal subgroup
of $\,G$, its index in $G$ is $\geqslant d$.
 \end{enumerate}
\end{corollary}

\begin{corollary} \label{inclCr}
If $\;{\rm Cr}_n$\;\lb resp.\;${\rm Aut}({\bf A}^{\hskip -.4mm n})\!)$ is non-Jordan, then
 for every integer $d >\!J_{\leqslant n}$, there is a finite subgroup $\,G$
of $\;{\rm Cr}_n$\;\lb resp.\;${\rm Aut}({\bf A}^{\hskip -.4mm n})\!)$ with the properties:
\begin{enumerate}[\hskip 4.2mm\rm(i)]
\item the action of $\,G$ on ${\bf A}^{\hskip -.4mm n}$ is
nonlinearizable;
 \item
 for any
abelian normal subgroup
of $\,G$, its index in $G$ is $\geqslant d$.
\end{enumerate}
\end{corollary}
\begin{proof} This follows from Corollary \ref{incl} because ${\bf GL}_n(k)$ is a connected affine algebraic subgroup of  ${\rm Aut}({\bf A}^{\hskip -.4mm n})$ and nonlinearizability of the action of $G$ on ${\bf A}^{\hskip -.4mm n}$ means that $G$ is not contained in a subgroup of $\,{\rm Cr}_n$\;\lb resp.\;${\rm Aut}({\bf A}^{\hskip -.4mm n})\!)$ conjugate to ${\bf GL}_n(k)$.
\quad $\square$ \end{proof}

\subsubsection{Recent developments}\label{recent}

The
initiated in \cite{Po1}
line of research
of Jor\-da\-ness
of ${\rm Aut}(X)$ and ${\rm Bir}(X)$
for algebraic varieties $X$
has generated interest of
algebraic geometers in Moscow among whom I have promoted it,
and led to a further
progress in
Problem B (hence A as well) in papers
\cite{Z}, \cite{PS1}, \cite{PS2}.\;In  \cite{Z}, the earliest of them,
the examples of non-Jordan groups ${\rm Bir}(X)$
 only known to date (October 2013)
have been constructed
(see
Theorem \ref{nJ} above).\;Below are formulated the results obtained in \cite{PS1}, \cite{PS2}.\;Some of them are conditional, valid under the assumption that the following  general
 conjecture by A.\,Borisov, V.\,Alexeev, and L.\,Borisov
 holds true:

\vskip 2mm

\noindent
{\it BAB Conjecture.}\;{\rm All Fano varieties of a fixed dimension $n$ and with terminal singularities are contained in a finite number of algebraic families.}
\vskip 2mm
\begin{theorem}[{{\rm \cite[Thm.\,1.8]{PS1}}}]\label{RC} If the\;BAB\;Conjecture holds true in dimension $n$, then, for every rationally connected $n$-dimensional variety $X$, the group ${\rm Bir}(X)$ is Jordan and, moreover, $J_{{\rm Bir}(X)}\leqslant u_n$ for a number $u_n$ depending only on $n$.
\end{theorem}

Since  for $n=3$ the BAB Conjecture is proved \cite{KMMT}, this yields
\begin{corollary}[{{\rm \cite[Cor.\,1.9]{PS1}}}] \label{Cr3} The space Cremona group
${\rm Cr}_3$ is Jordan.
\end{corollary}
 \begin{proposition}[{{\rm \cite[Prop.\;1.11]{PS1}}}] $u_3\leqslant (25920\cdot 20736)^{20736}$.
 \end{proposition}

 The pivotal idea of the proof of Theorem \ref{RC} is to use
 a technically refined
  form of
 the ``fixed-point method'' described in Subsection \ref{fpm}.

 \begin{theorem}[{{\rm \cite[Thm.\,1.8]{PS2}}}]\label{uniru} Let $X$ be an irreducible smooth proper $n$-dimensional variety.
 \begin{enumerate}[\hskip 2.2mm \rm (i)]
 \item The group $\,{\rm Bir}(X)$ is Jordan in either of the cases:
 \begin{enumerate}[\hskip .2mm \rm (a)]
 \item $X$ is non-uniruled;
 \item  the BAB Conjecture holds true in dimension $n$ and the irregula\-ri\-ty of $X$\linebreak {\lb}i.e., the dimension of its Picard variety\rb\,
 is\;$0$.
 \end{enumerate}
  \item If $\,X$ is non-uniruled and its irregularity is $0$, then the group $\;{\rm Bir}(X)$ is bo\-un\-ded
 {\lb}\!see Definition\;{\rm{\ref{b}}}\rb.
 \end{enumerate}
 \end{theorem}

\section{Appendix: Problems}

Below I add a few additional problems to those which have already been formulated above (Problems A and  B in Subsection \ref{se}, and Questions \ref{Jord},\;\ref{QB}).

\subsection{\boldmath ${\rm Cr}_n$-conjugacy of finite subgroups of \boldmath${\;\bf GL}_n(k)$}

 Below ${\bf GL}_n(k)$ is identified in the standard way with the subgroup of ${\rm Cr}_n$, which, in turn, is identified with the subgroup of ${\rm Cr}_m$ for any $m=n+1, n+2, \ldots, \infty$ (cf.\;\cite[Sect.\,1]{P3} or \cite[Sect.\,1]{P2}).

 \begin{question}\label{iso} Consider the following properties of two finite subgroups $A$ and $B$ of
 $\,{\bf GL}_n(k)$:
 \begin{enumerate}[\hskip 2.2mm \rm(i)]
\item $A$ and $B$ are isomorphic,
\item $A$ and $B$ are conjugate in ${\rm Cr}_n$.
 \end{enumerate}
 Does (i) imply (ii)?
 \end{question}

\noindent {\it Comments.}

 1.\;Direct verification based on the classification in \cite{DI} shows
that the answer is affirmative for $n\leqslant 2$.

 2.\;By \cite[Cor.\;5]{P2},  if $A$ and $B$ are abelian, then the answer is affirmative for every $n$.

 3. If $A$ and $B$ are isomorphic, then they are conjugate in ${\rm Cr}_{2n}$.\;This is the corollary of the following  stronger statement:

 \begin{proposition}\label{co}
  For any finite group $\,G$ and any injective homomorphisms
 \begin{equation}\label{alpha}
 \xymatrix@C=.9mm{
 G\ar@/^1.0pc/[rrr]^{\alpha_1}\ar@/_1.0pc/[rrr]_{\alpha_2} &&& {\bf GL}_n(k),
  }
 \end{equation}
   there exists an element $\varphi\in {\rm Cr}_{2n}$ such that $\alpha_1={\rm Int}(\varphi)\circ\alpha_2$.
 \end{proposition}
 \begin{proof}
  Every element $g\in {\bf GL}_n(k)$ is a linear transformation $x\mapsto g\cdot x$ of ${\mathbf A}^n$
  (with respect to the standard structure of $k$-linear space
  on ${\mathbf A}^n$).
 The injections $\alpha_1$ and $\alpha_2$ determine two faithful linear actions
  of $G$ on ${\mathbf A}^n$: the $i$th action ($i=1, 2$) maps $(g, x)\in G\times {\mathbf A}^n$ to
  $\alpha_i(g)\cdot x$.
 Consider the product of these actions, i.e., the action of $G$ on ${\mathbf A}^n\times {\mathbf A}^n$ defined by
 \begin{equation}\label{action}
 G\times {\mathbf A}^n\times {\mathbf A}^n\to {\mathbf A}^n\times {\mathbf A}^n,\quad (g, x, y)\mapsto (\alpha_1(g)\cdot x, \alpha_2(g)\cdot y).
 \end{equation}
 The natural projection of ${\mathbf A}^n\times {\mathbf A}^n\to {\mathbf A}^n$ to the $i$th factor is $G$-equivariant. By classical Speiser's Lemma (see
 \cite[Lemma 2.12]{LPR} and references therein),
 this implies that
${\mathbf A}^n\times {\mathbf A}^n$ endowed with $G$-action \eqref{action} is
 $G$-equivariantly birationally isomorphic to  ${\mathbf A}^n\times {\mathbf A}^n$
endowed with the  $G$-action via the $i$th factor by means of $\alpha_i$.\;Therefore, ${\mathbf A}^n\times {\mathbf A}^n$ endowed with the $G$-action via the first factor by means of $\alpha_1$ is $G$-equivariantly birationally isomorphic to
${\mathbf A}^n\times {\mathbf A}^n$ endowed with the $G$-action via the second factor by means of $\alpha_2$; whence the claim.
\quad $\square$ \end{proof}

 \begin{remark} In general, it is impossible to replace ${\rm Cr}_{2n}$ by ${\rm Cr}_n$
 in Proposition\;\ref{co}.\;Indeed, in \cite{RY} one finds the examples of finite abelian groups $G$ and embeddings \eqref{alpha} such that $\alpha_1\notin {\rm Int}({\rm Cr}_n)\circ\alpha_2$.\;However, since the ima\-ges of these embeddings are isomorphic finite abelian subgroups of ${\bf GL}_n(k)$,  by \cite[Cor.\;5]{P2} these images are conjugate in ${\rm Cr}_n$.
 \end{remark}

 \subsection{Torsion primes}

Let $X$ be an irreducible variety.\;The following definition
is based on the fact that the notion of algebraic torus in
${\rm Bir}(X)$ makes sense.

 \begin{definition}[{{\rm \cite[Sect.\;8]{P2}}}]
 Let $G$ be a subgroup of ${\rm Bir}(X)$.\;A prime integer $p$ is called a {\it torsion number} of $G$ if there exists a finite abelian $p$-subgroup of $G$ that does not lie in any torus of $G$.
 \end{definition}

Let ${\rm Tors}(G)$ be the set of all torsion primes of $G$.\;If $G$ is a connected reductive algebraic subgroup of ${\rm Bir}(X)$, this set coincides with that of the torsion primes of algebraic group $G$ in the sense of classical definition, cf., e.g.,\;\cite[1.3]{Serre1}.

\begin{question}[{\rm\cite[Quest.\,3]{P2}}]
What are, explicitly,
$${\rm Tors}({\rm Cr}_n), \;\;
{\rm Tors}({\rm Aut}\,{\bf A}^n),\;\;
{\rm Tors}({\rm Aut}^*{\bf A}^n),\quad n=3, 4,\ldots, \infty
$$
where ${\rm Aut}^*{\bf A}^n$ is the group of those automorphisms of ${\bf A}^n$ that preserve the volume form $dx_1\wedge\cdots\wedge dx_n$ on ${\bf A}^n$ (here $x_1,\ldots, x_n$ are the standard coordinate functions on ${\bf A}^n$), cf.\,\cite[\S2]{P2}?
  \end{question}

\noindent {\it Comments.}
By \cite[Sect.\,8]{P2},
 \begin{align*}
 {\rm Tors}({\rm Cr}_1)&=\{2\},\\
 {\rm Tors}({\rm Cr}_2)&
 =\{2, 3, 5\}\quad\mbox{(this coincides with ${\rm Tors}(E_8)$),}\\
  {\rm Tors}({\rm Cr}_n)&\supseteq\{2, 3\}\quad\mbox{ for any $n\geqslant 3$},\\
  {\rm Tors}({\rm Aut}\,{\bf A}^n)&=
{\rm Tors}({\rm Aut}^*{\bf A}^n)=\varnothing\quad \mbox{for $n\leqslant 2$}.
  \end{align*}

  \begin{question}[{\rm \cite[Quest.\,4]{P2}}] What is the minimal $n$ such that $7$ lies in one of the sets
 ${\rm Tors}({\rm Cr}_n)$, ${\rm Tors}({\rm Aut}{\bf A}^{\hskip -.3mm n})$, ${\rm Tors}({\rm Aut}^*{\bf A}^{\hskip -.3mm n})$?
  \end{question}

  \begin{question}[{\rm\cite[Quest.\,5]{P2}}]
  Are these sets finite?
  \end{question}

  \begin{question}
Are
the sets
$$\textstyle
\bigcup_{n\geqslant 1} {\rm Tors}({\rm Cr}_{n}),\quad \bigcup_{n\geqslant 1} {\rm Tors}({\rm Aut}{\bf A}^{\hskip -.3mm n}),\quad \bigcup_{n\geqslant 1} {\rm Tors}({\rm Aut}^*{\bf A}^{\hskip -.3mm n})$$
finite?
  \end{question}

  \subsection{Embeddability in \boldmath${\rm Bir}(X)$}
Not every group $G$ is embeddable in ${\rm Bir}(X)$ for some $X$.\;For instance, by \cite[Thm.\,1.2]{Co2},
if $G$ is finitely generated, its embeddability in ${\rm Bir}(X)$ implies that $G$
has a solvable word problem.\;Another example:  by \cite{Ca}, ${\bf PGL}_\infty(k)$ is not embeddable
in ${\rm Bir}(X)$ for $k\!=\!\mathbf C$ (I thank S.\;Cantat who
informed me in \cite{Ca13} about
these examples).

 If ${\rm Bir}(X)$ is Jordan, then by Example \ref{nJ2} and Theorem \ref{gs}(1)(i),
  $\mathcal N$ is not embeddable in
  ${\rm Bir}(X)$.\;Hence, by Theorem \ref{uniru}(i)(a), $\mathcal N$ is not embeddable in
  ${\rm Bir}(X)$ for any non-uniruled $X$.

  \begin{conjecture}\label{conj}  The finitely generated
  group $\mathcal N$  defined in Example \ref{nJ2}
  is not embeddable in
  ${\rm Bir}(X)$ for every irreducible variety $X$.
  \end{conjecture}

  Since $\mathcal N$ contains ${\rm Sym}_n$ for every $n$, and every finite group can be embedded in
  ${\rm Sym}_n$ for an appropriate $n$, the existence of an irreducible variety $X$ for which ${\rm Bir}(X)$ contains an isomorphic copy of $\mathcal N$ implies that ${\rm Bir}(X)$ contains an isomorphic copy of every finite group and, in particular, every simple finite group. Therefore, Conjecture \ref{conj} follows from the affirmative answer to

   \begin{question}\label{co} Let $X$ be an irreducible variety. Is
   any set of pairwise nonisomorphic simple nonabelian finite subgroups of ${\rm Bir}(X)$ finite?
\end{question}

The affirmative answer looks likely.
At this writing (October 2013) about this question I know the following:

\begin{proposition}
If $\dim X\leqslant 2$, then the answer to Question \ref{co} is affirmative.
\end{proposition}
\begin{proof}
The  claim immediately follows from Theorems \ref{CS} and \ref{simple}
if $X$ is not birationally isomorphic to ${\bf P}^1\times B$, where $B$ is an elliptic curve. For $X={\bf P}^1\times B$ the claim follows from \eqref{semi}   and Theorem \ref{simple} because
${\rm Aut}(B)$ and ${\rm Bir}_B(X)$ are Jordan groups.
\quad $\square$
\end{proof}

By Theorems \ref{RC}, \ref{uniru} and by \cite{KMMT}, the answer to Question \ref{co} is affirmative also in each of the following cases:
 \begin{enumerate}[\hskip 4.2mm \rm(i)]
 \item $X$ is non-uniruled;
 \item $X$ is rationally connected   or smooth proper with irregula\-ri\-ty $0$, and
 \begin{enumerate}[\hskip -.9mm\rm(a)]
 \item either $\dim X=3$ or
 \item $\dim X>3$ and  the\;BAB\;Conjec\-ture holds true in dimension $\dim X$.
 \end{enumerate}
 \end{enumerate}

 Note that  if  $X$ and ${\rm Bir}(X)$ in Question \ref{co} are replaced, respectively, by a connected smooth topological ma\-ni\-fold $M$ and ${\rm Diff}(M)$, then by Theorem \ref{Man}, for a noncompact
 $M$,
 the answer,
 in general, is negative. But for a compact $M$ a finiteness theorem \cite[Thm.\,2]{Po5} holds.

 \subsection{Contractions}  Developing  the classical line of research,  in recent ye\-ars
 were growing activities aimed at description  of finite subgroups of ${\rm Bir}(X)$ for various $X$; the case of
rational $X$ (i.e., that of the Cremona group ${\rm Bir}(X)$) was,
probably, most actively explored with culmination in the classification of finite subgroups of
${\rm Cr}_2$, \cite{DI}.\;In these studies, all the classified groups  appear  in the corresponding lists on equal footing.\;However, in fact, some of them are ``more basic'' than the others
because the latter may be obtained from the former by  a certain standard construction.\;Given this, it is natural to pose the problem of describing these ``basic'' groups.

Namely,  let $X_1$ and $X_2$ be the irreducible varieties
and let $G_i\subset {\rm Bir}(X_i)$, $i=1, 2$, be the subgroups isomorphic
to a finite group $G$.\;Assume that fixing the isomorphisms $G\to G_i$, $i=1, 2$,
we obtain the rational actions of $G$ on $X_1$ and $X_2$ such that there is
a $G$-equivariant rational dominant map $\varphi\colon X_1\dashrightarrow X_2$.\;Let
$\pi^{}_{X_i}\colon X_i\dashrightarrow  X_i\dss G$, $i=1, 2$  be
the rational quotients (see,\,e.g.\;\cite[Sect.\,1]{P3}) and let
$\varphi^{}_{G}\colon X_1\dss G\dashrightarrow X_2\dss G$ be
 the do\-mi\-nant rational
map induced by
$\varphi$.\;Then the following holds (see,\;e.g.\;\cite[Sect.\;2.6]{Re}):

(1) The appearing commutative diagram
\begin{equation}\label{diag}
\begin{matrix}
\xymatrix{X_1\ar@{-->}[r]^{\varphi}\ar@{-->}[d]_{\pi^{}_{X_1}}&X_2\ar@{-->}[d]^{\pi^{}_{X_2}}\\
X_1\dss G\ar@{-->}[r]^{\varphi^{}_{G}}& X_2\dss G
}
\end{matrix}
\end{equation}
is, in fact, cartesian, i.e.,
$\pi^{}_{X_1}\colon X_1\dashrightarrow  X_1\dss G$ is obtained from $\pi^{}_{X_2}\colon X_2\dashrightarrow X_2\dss G$ by the base change  $\varphi^{}_{G}$.\;In particular,  $X_1$ is birationally
$G$-isomorphic to
$$X_2\times _{X_2\dss G}(X_1\dss G):=\overline{\{(x, y)\in X_2\times (X_1\dss G) \mid
\pi^{}_{X_2}(x)=\varphi^{}_{G}(y)\}}.$$

(2) For every irreducible variety $Z$ and every dominant
rational map $\beta\colon Z\to X_2\dss G$ such that $X_2\times_{X_2\dss G}Z$
is irreducible, the variety $X_2\times_{X_2\dss G}Z$ inherits via $X_2$ a faithful rational action of $G$  such that
one obtains commutative diagram \eqref{diag} with $X_1=X_2\times_{X_2\dss G}Z$, $\varphi^{}_{G}=\beta$, and $\varphi={\rm pr}_1$.

If such a $\varphi$ exists,
we say that $G_1$ is {\it induced} from $G_2$ {\it by a base change}. The latter is called {\it trivial} if
$\varphi$ is a birational isomorphism.\;If a finite subgroup $G$ of ${\rm Bir}(X)$ is not induced by a nontrivial base change, we say that $G$ is {\it incompressible}.

\begin{example}
The standard embedding ${\rm Cr}_n \hookrightarrow {\rm Cr}_{n+1}$ permits to consider the finite subgroups of ${\rm Cr}_n$ as that of ${\rm Cr}_{n+1}$.\;Every finite subgroup of ${\rm Cr}_{n+1}$ obtained this way is induced by the nontrivial base change determined by the projection ${\bf A}^{\hskip -.4mm n+1}\to
{\bf A}^{\hskip -.4mm n}$, $(a_1,\ldots, a_n, a_{n+1})\mapsto (a_1,\ldots, a_n, a_{n+1})$.
\quad $\square$ \end{example}

\begin{example}
This is Example 6 in \cite{Re2}.\;Let $G$ be a finite group that does not embed in ${\rm Bir}(Z)$ for any curve $Z$ of genus $\leqslant 1$ (for instance, $G={\rm Sym}_5$) and let $X$ be a smooth projective
curve of minimal
possible genus such that $G$ is isomorphic to a subgroup of ${\rm Aut}(X)$.\;Then
this subgroup of ${\rm Bir}(X)$
is incompressible.\hskip -1mm\footnote{The proof in  \cite{Re2} should be corrected as follows.\;Assume that there is a faithful
action of $G$ of a smooth projective curve $Y$ and a dominant $G$-equivariant
morphism $\varphi\colon X\to Y$ of degree $n>1$.\;By the construction,  $X$ and $Y$ have the same genus $g>1$, and the Hurwitz formula yields
that the number of branch points of $\varphi$ (counted with positive multiplicities) is the integer $(n-1)(2-2g)$.\;But the latter is negative,\,---\,a contradiction.}\quad $\square$
\end{example}

\begin{example} By Example 5 in \cite{Re2},
a finite cyclic subgroup of order $\geqslant 2$ in ${\rm Bir}(X)$ is never incompressible.
\quad $\square$
\end{example}

\begin{example}\label{G2}
Consider two rational actions of
$G:={\rm Sym}_3\times \mathbb Z/2\mathbb Z$ on ${\bf A}^{\hskip -.4mm 3}$.\;The subgroup  ${\rm Sym}_3$ acts by natural permuting the coordinates in both cases.\;The nontrivial element of $Z/2\mathbb Z$ acts  by
$(a_1, a_2, a_3)\mapsto (-a_1, -a_2, -a_3)$ in the first case and by $(a_1, a_2, a_3)\mapsto (a_1^{-1}, a_2^{-1}, a_3^{-1})$ in the second.\;The surfaces
\begin{align*}
P&:=\{(a_1, a_2, a_3)\in {\bf A}^{\hskip -.4mm 3}\mid
a_1 + a_2 + a_3=0\},\\[-.6mm]
T&:=\{(a_1, a_2, a_3)\in {\bf A}^{\hskip -.4mm 3}\mid
a_1 a_2 a_3=1\}
\end{align*}
 are $G$-stable in, resp., the first and the second case.\;Since $P$ and $T$ are rational,
these actions of $G$ on $P$ and $T$ determine, up to conjugacy, resp., the subgroups $G_P$ and $G_T$ of ${\rm Cr}_2$, both isomorphic to $G$.\;By \cite{isk} (see also\;\cite{LPR}, \cite{LPZ2}),  these subgroups
are not conjugate in  ${\rm Cr}_2$.\;How\-ever, by \cite[Sect.\;5]{LPZ2},    $G_T$ is induced from
$G_P$ by a nontrivial base change (of degree $2$).
\quad $\square$ \end{example}

In fact, Example \ref{G2} is a special case (related to the simple algebraic group ${\bf G}_2$) of the following

\begin{example}\label{Caley} Let $G$ be a connected reductive algebraic group. Re\-call \cite[Def.\,1.5]{LPR} that $G$ is called a Cayley group if there is a birational isomorphism of $\lambda\colon G\dashrightarrow {\rm Lie}(G)$, where ${\rm Lie}(G)$ is the Lie algebra of $G$, equivariant with respect to the conjugating and adjoint actions of $G$ on the underlying varieties of $G$ and ${\rm Lie}(G)$, respectively, i.e., such that
\begin{equation}\label{Cay}
\lambda(gXg^{-1})={\rm Ad}_G^{}g(\lambda(X))
\end{equation}
if $g$ and $X\in G$ and both sides of \eqref{Cay} are defined.

Fix  a maximal torus $T$ of $G$ and consider the natural actions of the Weyl group $W=N_G(T)/T$ on $T$ and on  $\mathfrak t:={\rm Lie}(T)$.\;Since these actions are faithful and $T$ and $\mathfrak t$ are rational varieties, this determines, up to conjugacy, two embeddings  of $W$ in  ${\rm Cr}_r$, where $r=\dim T$.\;Let $W_T$ and $W_{\mathfrak t}$ be the images of these embeddings.\;By \cite[Lemma\;3.5(a) and Sect.\,1.5]{LPR}, if $G$ is not Cayley and $W$ has no outer autormorphisms, then $W_T$ and $W_{\mathfrak t}$ are not conjugate in
${\rm Cr}_r$.\;On the other hand, by \cite[Lemma\;10.3]{LPR}, $W_T$ is induced from $W_{\mathfrak t}$ by a (nontrivial) base change (see also Lemma \ref{based} below).

This yields, for arbitrary $n$, the examples of isomorphic nonconjugate finite subgroups of ${\rm Cr}_n$ one of which is induced from the other by a nontrivial base change.\;For instance, if $G\!=\!{\bf SL}_{n+1}$, then $r\!=\!n$ and $W={\rm Sym}_n$.\;Since, by \cite[Thm.\;1.31]{LPR}, ${\bf SL}_{n+1}$ is not Cayley for $n\geqslant 3$ and ${\rm Sym}_n$ has no outer automorphisms for $n\neq 6$, the above construction yields
for these $n$ two nonconjugate subgroups of ${\rm Cr}_n$ isomorphic to ${\rm Sym}_n$, one of which is induced from the other by a nontrivial base change.
\quad $\square$\end{example}

The following gives a general way of constructing two finite subgroups of ${\rm Cr}_n$ one of which is induced from the other by a base change.

Consider an $n$-dimensional irreducible nonsingular variety $X$ and  a finite subgroup $G$ of ${\rm Aut}(X)$.\;Suppose that $x\in X$ is a fixed point of $G$.\;By Lemma \ref{linea}, the induced action of $G$ on the tangent space of $X$ at $x$ is faithful.\;Therefore this action determines, up to conjugacy, a subgroup $G_1$ of ${\rm Cr}_n$ isomorphic to $G$. On the other hand, if $X$ is rational, the action of $G$ on $X$ determines,
up to conjugacy, another subgroup $G_2$ of  ${\rm Cr}_n$ isomorphic to $G$.
\begin{lemma}\label{based}
$G_2$ is induced from $G_1$ by a base change.
\end{lemma}
\begin{proof}
By Lemma \ref{cov} we may assume that $X$ is affine, in which case the claim follows from \cite[Lemma\;10.3]{LPR}.
\quad $\square$ \end{proof}

\begin{corollary} Let $X$ be a nonrational irreducible variety and
let $G$ be an incompressible finite subgroup of
${\rm Aut}(X)$.\;Then $X^G=\varnothing$.
\end{corollary}

\begin{question}
Which finite subgroups of ${\rm Cr}_2$ are incompressible?
\end{question}



 \end{document}